\documentclass[11pt,reqno]{amsart}
\usepackage{amssymb, mathtools, diagrams}
\usepackage[pagewise]{lineno}
\numberwithin{equation}{section}
\usepackage{cite}
\usepackage{url}
\usepackage{graphicx,lineno}
\usepackage[all]{xy}

\theoremstyle{definition}
\newtheorem{thm}{Theorem}[section]
\newtheorem{cor}[thm]{Corollary}
\newtheorem{lem}[thm]{Lemma}
\newtheorem{exa}[thm]{Example}
\newtheorem{prop}[thm]{Proposition}
\newtheorem{defi}[thm]{Definition}
\newtheorem{rem}[thm]{Remark}

\DeclareMathOperator{\Hc}{\mathcal{H}om}

\DeclareMathOperator{\Ext}{\mathrm{Ext}}

\DeclareMathOperator{\N}{\mathcal{N}}

\DeclareMathOperator{\I}{\mathcal{I}}
\DeclareMathOperator{\mo}{\mathcal{O}}

\newcommand{\wt}[1]{\widetilde{#1}}

\newcommand{\mr}[1]{\mathrm{#1}}
\newcommand{\mb}[1]{\mathbb{#1}}

\newcommand{\mc}[1]{\mathcal{#1}}
\newcommand{\ov}[1]{\overline{#1}}

\newcommand{\mf}[1]{\mathfrak{#1}}

\newcommand{\co}{\mathcal{O}}

\begin{document}

 \title{Cohen-Macaulay modules and the Bondal-Orlov conjecture}

\author{Ananyo Dan}

\address{CUNEF Universidad, C. de Leonardo Prieto Castro, 2, Moncloa - Aravaca, 28040 Madrid, Spain}

\email{dan.ananyo@cunef.edu}

\author{Yirui Xiong}

\address{School of Sciences, Southwest Petroleum University, 610500 Chengdu, People's Republic of China}

\email{yiruimee@gmail.com}



\date{\today}

\begin{abstract}
Most of the known examples of derived categories of small resolutions arise as the 
derived category of the endormorphism algebra of tilting bundles or complexes.
Given two resolutions connected by a flop, if the strict transform of a tilting bundle is again tilting, 
then the derived categories of the two resolutions are equivalent, thereby proving the Bondal-Orlov conjecture in this setup.
Unfortunately, it is difficult to produce tilting bundles that are compatible with flops.
In this article, we introduce the notion of CM-degree of locally-free sheaves on resolutions and use them to construct tilting generators.
In particular, we show that if there exists a relative very ample line bundle on the resolution with 
CM degree equal to the dimension of the exceptional locus, then the generator bundles constructed by
Van den Bergh \cite{van1} and Toda-Uehara \cite{toda2} are also tilting bundles.
The advantage of our approach is that the CM degree is preserved under strict transform.
As a consequence we prove the Bondal-Orlov conjecture in certain cases of small resolutions.
\end{abstract}

 \maketitle

 \section{Introduction}
 The underlying field will always be $\mb{C}$.
 A conjecture of Bondal and Orlov \cite{bondor1}, claims that any two crepant resolutions 
 of a singular variety  are derived equivalent.
 Recall, a resolution is \emph{crepant} if the canonical divisor of the resolution is the  pullback of 
 the canonical sheaf on the singularity. 
 The conjecture was proved when the singular variety is three dimensional by Bridgeland \cite{brid1} and using different methods Van den Bergh \cite{van1}.
 The key idea in \cite{van1} 
 is to find a locally-free sheaf (or a complex) on the resolution, called the \emph{tilting generator} such that the 
 derived category of the resolution is derived equivalent to the derived category of the endomorphism algebra of the tilting generator.
 In the case of a three dimensional singularity, the tilting generator is a 
 direct sum of a very ample line bundle and some of its tensor powers.
Toda and Uehara \cite{toda2} gave a general strategy to produce tilting generators in small (crepant) resolutions of any dimension,
 under a cohomological assumption. In particular, they showed that given a small (crepant) resolution with exceptional locus $E$
 and a globally generated ample line bundle $\mc{L}$ on the resolution if $H^i(\mc{L}^{-j})=0$ for all $i>1$ and $0<j<\dim(E)$,
 then there is an iterative method to produce a tilting generator. When $\dim(E)=1$, this gives back the 
 tilting generator of Van den Bergh.

 In this article, we study the generator given by Van den Bergh \cite{van1} and give a necessary and sufficient condition for this 
  to be a tilting bundle in any dimension and more generally for isolated Cohen-Macaulay singularities (not necessarily Gorenstein).
  In this setup, the pullback of the dualizing sheaf from the singularity to the resolution 
  is not necessarily an invertible sheaf, much less isomorphic to the canonical bundle of the resolution.
  We use the definitions in Toda-Uehara \cite{toda2}: by \emph{tilting bundle} on a resolution $\wt{X}$, 
 we mean a locally-free $\mo_{\wt{X}}$-module $\mc{E}$ such that $\mr{Ext}^i_{\wt{X}}(\mc{E},\mc{E})=0$ for all $i>0$ (this is called the
 \emph{partial tilting bundle} in \cite{van-ar, iyama3}).
 In particular, $\mc{E}$ need not be a generator. 
Our objectives in this article are:

 \begin{enumerate}
     \item Find line bundles $\mc{L}$ over small resolutions such that the cohomological vanishing conditions of Toda and Uehara are satisfied (Proposition \ref{prop:van} and Theorem \ref{thm:van-main}).
     \item Produce tilting bundles over small resolutions of isolated Cohen-Macaulay singularities (not necessarily Gorenstein).
     See Corollary \ref{cor:tb}. 
     If this tilting bundle is compatible with small resolutions related by a flop, then we get new instances of the Bondal-Orlov conjecture (Theorem \ref{thm:orlov}).
      \end{enumerate}

The key observation that helps us fulfil objective $(1)$ is that if the pushforward to $X$ of an invertible $\mo_{\wt{X}}$-module $\mc{L}$ is
maximal Cohen-Macaulay and $\mc{L}^{\vee}$ is very ample, then all the higher cohomologies of $\mc{L}$ vanish. See Proposition \ref{prop:van}
for a more general result. We need the following terminology. 
A locally-free sheaf $\mc{F}$ on $\wt{X}$ is called \emph{almost-full} if $\pi_*\mc{F}$ is a maximal Cohen-Macaulay module.
The terminology arises from the literature on McKay correspondence (see \cite[Definition $1.1$]{kahn} or \cite[Definition $3.3$]{BoRo}).
We say that $\mc{F}$ \emph{has CM-degree} $m$, if there exist $m$ sections $s_1,...,s_m \in H^0(\mc{F})$ such that the image under $\pi$ 
of the common-zero locus of $i$ sections $s_1,...,s_i$ is Cohen-Macaulay of codimension $i$, for all $1 \le i \le m$ (see  Definition \ref{defi:cm-deg}).
It is easy to check that $\mc{F}$ has CM-degree $1$ if and only if $\mc{F}^{\vee}$ is almost full.
In Theorem \ref{thm:mck} we give a cohomological description of  almost full sheaves of rank one. This plays a vital role in 
Theorem \ref{thm:tens} to check when is the tensor product of two almost full sheaves, again almost full.
Using this we prove the following vanishing theorem:

\begin{thm}\label{thm:intro-1}
    Let $(X,x)$ be an isolated Cohen-Macaulay singularity of dimension at least $3$ and $\wt{X}$  
    be a small resolution of $X$. If $\mc{L}$ is an invertible sheaf on $\wt{X}$ with CM-degree $m$, 
    then  $H^i((\mc{L}^{\vee})^{\otimes j})=0$ for all $i>0$ and $1 \le j \le m$.
\end{thm}

See Theorem \ref{thm:van-main} for a proof. This implies that if $\mc{L}$ is of CM-degree $m$, then the vector bundle of the form 
$\mc{E}_l:=\mo_{\wt{X}} \oplus \bigoplus\limits_{i=1}^l \mc{L}^{\otimes i}$ is a tilting bundle for all $l \le m$ (Corollary \ref{cor:tb}). 
Conversely, if a vector bundle of the form $\mc{E}_l$ is a tilting bundle, then $\mc{L}$ is of CM-degree $l$ (Corollary \ref{lem:tilt-2}).
Alternatively, if $\mc{L}$ is a very ample line bundle of CM-degree equal to $\dim(E)$, then one can use the construction of 
Toda-Uehara to get a different tilting bundle. 

One of the challenges to proving the Bondal-Orlov conjecture is that tilting bundles are not preserved under flops.
In contrast, the strict transform of a CM-degree $m$ line bundle, is again a line bundle of CM-degree $m$.
As a result, any tilting bundle that can be constructed using line bundles of CM-degree equal to $\dim(E)$ 
(for example, the one of Toda-Uehara or Van den Bergh) can give us a derived equivalence between resolutions related by a flop.
 We prove:

\begin{thm}
    Let $(X,x)$ be an isolated Gorenstein singularity of dimension at least $3$ and 
    \[\pi^+: \wt{X}^+ \to X\, \mbox{ and } \pi^-: \wt{X}^- \to X\]
    be two small resolutions of $X$, which is a $D$-flop in the sense that $D$ is a very ample divisor on $\wt{X}^+$
    and its strict transform in $\wt{X}^-$ is anti-ample (i.e., dual of an ample divisor). Suppose that the dimension of the 
    exceptional locus of $\pi^+$ and $\pi^-$ are the same, say equal to $n$. If  $\mc{L}^+:=\mo_{\wt{X}^+}(D)$
    is of CM-degree $n$, then 
    \begin{enumerate}
        \item  
        $\mc{E}^+:=\bigoplus_{j=0}^n (\mc{L}^+)^{\otimes j}$ is a tilting generator on $\wt{X}^+$.
        \item The dual of the strict transform in $\wt{X}^-$ of $\mc{E}^+$ is again a tilting generator.
        \item The tilting generators induce a derived equivalence:
        $D^b(\wt{X}^+) \cong D^b(\wt{X}^-)$.
    \end{enumerate}
\end{thm}

See Theorem \ref{thm:orlov} for a proof and \S \ref{sec:mukai} for an application to the Mukai flop.
As a byproduct we study the relation between tilting bundles and certain ``factorizations" of maximal Cohen-Macaulay modules.
 In particular, given a maximal Cohen-Macaulay module $M$ over an isolated Cohen-Macaulay scheme $(X,x)$, 
 we say that $M$ \emph{admits a tilting factorization} if there exists 
 a reflexive $\mo_X$-module $M_0$ such that $M \cong \mr{Hom}_X(M_0,M_0)$. We call $M_0$ a \emph{tilting factor}.
 See \S \ref{sec:exa-tilt} for all the tilting factors of the structure sheaf.
 We show that every tilting bundle gives rise to maximal Cohen-Macaulay modules admitting a tilting factorization.
 The converse holds true if the exceptional locus is of low dimension.  
 More precisely, we show:

 \begin{thm}\label{thm:intro}
     Let $(X,x)$ be an isolated CM-singularity of dimension $d$ and $\pi: \wt{X} \to X$ be a small resolution with local complete intersection 
     exceptional locus $E$. 
     If $X$ is Gorenstein and $\mc{E}$ is a tilting bundle on $\wt{X}$, then  
    $\pi_*(\mc{E} \otimes \mc{E}^{\vee})$ is a maximal Cohen-Macaulay $\mo_X$-module admitting 
    a tilting factorization with tilting factor $\pi_*(\mc{E})$.
    Conversely, if $d$ is odd, $\dim(E) \le d/2$, $M$ is a maximal Cohen-Macaulay $\mo_X$-module with tilting factor $M_0$, 
    then $\mc{M}_0:=(\pi^*M_0)^{\vee \vee}$ is a tilting bundle provided $\mc{M}_0$ is locally-free.
 \end{thm}

 See Theorem \ref{thm:tf} for a proof and a more general statement.  
 We note that, tilting factors have also been studied by Iyama-Wemyss \cite{iyama2, iyama3}, where they are called \emph{modifying modules}.
 We use the terminology tilting factors   
 to highlight similarities with tilting bundles on the resolutions. 
 Iyama and Wemyss show that if $X$ is a $3$ dimensional isolated Gorenstein singularity, then there is a $1-1$ correspondence 
 between modifying modules on $X$ and certain types of reflexive $\mo_X$-modules that are non-commutative analogues of tilting bundles.
 Theorem \ref{thm:intro} retrieves this result and generalizes it to the higher dimension case as well as when the singularity is 
 Cohen-Macaulay (not necessarily Gorenstein).  
 The advantage of the correspondence given in Theorem \ref{thm:intro} 
 is that, in comparison to tilting bundles, maximal Cohen-Macaulay $\mo_X$-modules are much easier to produce (see Examples \ref{exa:af}, \ref{example:ab}).
 They arise naturally as higher syzygies of torsion-free sheaves on a curve (see \cite[\S $1.3$]{brun1}).
 Moreover, if $M$ is a direct sum of rank one sheaves, then the reflexive hull $\mc{E}$ of the pull-back of $M$ is always locally-free.
 
 Finally, we note that Kaledin \cite{kale} proved that tilting generators exist in the case of symplectic singularities of any dimension. However, it is difficult 
 to extend his methods to non-symplectic singularities. 
 In higher dimension, tilting generators play a crucial role in obtaining non-commutative crepant resolutions as demonstrated in the works of 
 Iyama-Wemyss \cite{iyama1, iyama2, iyama3}, Wemyss-Donovon \cite{dono},
 \u{S}penko-Van den Bergh \cite{vspenko, vspenko-1, vspenko-2} and many others \cite{buch-3, van1}. 
 Although there are many similarities between non-commutative crepant resolutions and crepant resolutions, they often behave very differently.
 For example, even if there exists a crepant resolution, a non-commutative crepant resolution need not exist \cite[Example $3.5$]{dao}.
 Conversely, existence of a non-commutative crepant resolution does not imply existence of a commutative resolution (example, terminal cyclic 
 quotient singularity).




{\bf{Acknowledgements}:} We would like to thank Prof. Tom Bridgeland for his many helpful conversations and 
feedback on an earlier draft.

\section{Preliminaries}

\subsection{Local cohomology}
Let $X$ be a topological space and $\mc{F}$ be a sheaf of abelian groups on $X$. 
Given a closed subset $T \subset X$, denote by $\Gamma_T(X,\mc{F})$ the group of sections of $\mc{F}$
with support in $T$. Note that, $\Gamma_T(X,-)$ is a left exact functor. Denote by $H^i_T(X,-)$ the right derived functor 
of $\Gamma_T(X,-)$. Denote by $\mc{H}^i_T(X,\mc{F})$ the sheaf associated to the presheaf which assigns to an open subset $U \subset X$,
the group $H^i_{T \cap U}(T \cap U,\mc{F}|_U)$. We recall the following long exact sequences:

\begin{prop}\label{prop:loc}
 Let $T \subset X$ be a closed subvariety and $U:= X \backslash T$. Then, we have an exact sequence of the form:
 \[ .... H^{i}_T(\mc{F}) \to H^i(\mc{F}) \to H^i(U, \mc{F}|_U) \to H^{i+1}_T(\mc{F}) \to ... \]
 Furthermore, we have $R^i j_*\mc{F} \cong \mc{H}^{i+1}_T(\mc{F})$ for $i>0$ and 
 \[0 \to \mc{H}^0_T(\mc{F}) \to \mc{F} \to j_*(\mc{F}|_U) \to \mc{H}^1_T(\mc{F}) \to 0,\]
where $j: U \to X$ is the natural inclusion.
 \end{prop}

\begin{proof}
 See \cite[Corollary $1.9$]{grh}.
\end{proof}

\subsection{Duality theorem}
Let $X$ be the spectrum of a local noetherian $\mb{C}$-algebra with at worst isolated singularity and $\dim(X)=n$. Let 
\[\pi: \wt{X} \to X\]
be a resolution of singularities. Denote by $E$ the exceptional locus (not necessarily a divisor). 
Recall, the following useful duality theorem on local cohomology groups:

\begin{thm}\label{thm:locdual}
 Let $\mc{F}$ be a locally-free sheaf on $\wt{X}$. Then, Serre duality induces an isomorphism:
 \[H^i_E(\wt{X},\mc{F}) \cong H^{n-i}(\wt{X},\mc{F}^{\vee} \otimes K_{\wt{X}})^{\vee},\, \mbox{ for } 0<i< n.\]
\end{thm}

\begin{proof}
 See \cite[Theorem $3.2$]{karr} or \cite[Corollary $3.5.15$]{ishbook}.
\end{proof}

\section{Tilting factorization}

\subsection{Setup}\label{se01}
Let $A$ be a local, Cohen-Macaulay $\mathbb{C}$-algebra with maximal ideal $\mf{m}$. Suppose that for every prime ideal $\mf{p}$
of $A$, $A_{\mf{p}}$ is regular if $\mf{p} \not= \mf{m}$. In other words, $X=\mr{Spec}(A)$ has only isolated singularities.
Denote by $x$ the singular point of $X$ and $U:=X \backslash \{x\}$.

    \begin{rem}\label{rem:crep}
    Let $\pi: \wt{X} \to X$ be a small resolution. 
    Since $X$ is Cohen-Macaulay, it is equipped with a dualizing sheaf $\omega_X$.
    Note that $K_{\wt{X}} \cong (\pi^*\omega_X)^{\vee \vee}$, where $K_{\wt{X}}$ denotes 
    the canonical sheaf on $\wt{X}$. 
        It is well-known that in this case $R\pi_*\mo_{\wt{X}} \cong \mo_X$. Indeed, by the duality theorem
        \[R\pi_*\mo_{\wt{X}}=R\pi_* R\Hc_{\wt{X}}(K_{\wt{X}}, K_{\wt{X}})=R\Hc_X(R\pi_*K_{\wt{X}}, \omega_X). \]
        By Grauert-Riemenschneider theorem, $R\pi_*K_{\wt{X}}=\pi_*K_{\wt{X}}=\omega_X$, where the last equality 
        follows from the fact that $\pi$ is a small resolution. 
        Therefore, 
        \[R\Hc_X(R\pi_*K_{\wt{X}}, \omega_X)=R\Hc_X(\omega_X,\omega_X)=\Hc_X(\omega_X,\omega_X)=\mo_X,\]
        where the last two equalities follow from \cite[Theorem $3.3.10$]{brun1}. Note that, here $X$ need not be 
        Gorenstein.
    \end{rem}

\begin{defi}
    Let $M$ be a maximal Cohen-Macaulay $\mo_X$-module. We say that $M$ \emph{admits a tilting factorization} if there 
    exists a reflexive $\mo_X$-module $M_0$ such that $M \cong \mr{Hom}_X(M_0,M_0)$. If $M$ admits a tilting factorization
   as above, then $M_0$ is called a \emph{tilting factor} of $M$.
\end{defi}

\subsection{Example (Tilting factors of $\mo_X$):}\label{sec:exa-tilt}
Any tilting factor of $\mo_X$ is of the form $\I_{Z|X}$ or $\Hc_X(\I_{Z|X}, \mo_X)$ for some integral divisor $Z \subset X$. Indeed,   
let $\mc{F}$ be a tilting factor of $\mo_X$. Since $\mo_X$ is of rank one, $\mr{rk}(\mc{F})=1$. 
    For a general section $s$ of $\mc{F}$, we have a short exact sequence:
    \begin{equation}\label{eq:rk1}
        0 \to \mo_X \xrightarrow{s} \mc{F} \to \mc{A} \to 0,
    \end{equation}
    where $\mc{A}$ is supported on an integral divisor $Z \subset X$ and is torsion-free of rank one as an $\mo_Z$-module.
    Dualizing this, we get an exact sequence:
   \[        0 \to \mr{Hom}_X(\mc{F},\mo_X) \to \mo_X \to \mr{Ext}^1_X(\mc{A},\mo_X) \xrightarrow{\phi} \mr{Ext}^1_X(\mc{F},\mo_X) \to 0\]
    Denote by $\mc{A}':= \ker(\phi)$. Since $\mc{F}$ is locally-free away from the singular point $x$, $\mr{Ext}^1_X(\mc{F},\mo_X)$
    is supported in $x$. Therefore, the support of $\mc{A}'$ is the same as the support of $\mc{A}$, which is $Z$.
    Since there is a surjective morphism from $\mo_X$ to $\mc{A}'$, $\mc{A}' \cong \mo_Z$ i.e., we have a short exact sequence:
       \[ 0 \to \mr{Hom}_X(\mc{F},\mo_X) \to \mo_X \to \mo_Z \to 0.\]
    This implies $\mr{Hom}_X(\mc{F},\mo_X) \cong \I_{Z|X}$. In particular, $\I_{Z|X}$ is a reflexive sheaf. 
    Since $\mc{F}$ is reflexive, we have 
    \[\mc{F} \cong (\mc{F})^{\vee \vee} \cong \mr{Hom}_X(\I_{Z|X},\mo_X).\]
    Since dual of tilting factors are also tilting, $\I_{Z|X}$ is also a tilting factor of $\mo_X$. This proves our claim.

    Conversely, given an integral divisor $Z \subset X$, $\I_{Z|X}$ and $\Hc_X(\I_{Z|X},\mo_X)$ are tilting factors.
    Indeed, since $X$ is regular away from $x$, $\I_{Z|X}$ is an invertible sheaf away from $x$.
    Therefore, $\mr{Hom}(\I_{Z|X}, \I_{Z|X})$ is isomorphic to $\mo_X$, away from $x$. Since $\mo_X$ and $\mr{Hom}(\I_{Z|X},\I_{Z|X})$
    are reflexive sheaves that agree over a complement  of codimension at least $2$, we have 
    \[\mo_X \cong \mr{Hom}(\I_{Z|X},\I_{Z|X}) \cong \mr{Hom}_X(\I_{Z|X}^{\vee}, \I_{Z|X}^{\vee}).\]
    Therefore, $\I_{Z|X}$ and $(\I_{Z|X})^{\vee}$ are tilting factor of $\mo_X$. This proves our claim.

\subsection{Tilting bundle}
Setup as in \S \ref{se01}. Let $d:=\dim(X)$ and \[\pi: \wt{X} \to X\] be a small resolution.
Denote by $E$ the exceptional locus and $n:=\dim(E)$. 
Recall, a \emph{tilting bundle} on $\wt{X}$ is a locally-free $\mo_{\wt{X}}$-module $\mc{E}$
such that $H^i(\mc{E} \otimes \mc{E}^{\vee})=0$ for all $i>0$.
We prove:

\begin{thm}\label{thm:tf}
 The following are true:
 \begin{enumerate}
     \item  Let $\mc{E}$ be a tilting bundle on $\wt{X}$. If $X$ is Gorenstein, then  
    $\pi_*(\mc{E} \otimes \mc{E}^{\vee})$ is a maximal Cohen-Macaulay $\mo_X$-module admitting 
    a tilting factorization with tilting factor $\pi_*(\mc{E})$\footnote{In fact, we proved $\pi_*(\mc{E})$ is a noncommutative crepant resolution of $X$ if $\mc{E}$ is moreover a generator in $D^b(\wt{X})$, we thank      Jieheng Zeng for pointing out this to us.}.
    \item Let $M$ be a maximal Cohen-Macaulay $\mo_X$-module admitting a tilting factorization. Let $M_0$ be a tilting factor of $M$.
    Suppose that $E \subset \wt{X}$ is a local complete intersection and $\dim(E) \le d/2$. 
    If $d$ is odd and $\mc{M}_0:= (\pi^*M_0)^{\vee \vee}$ is locally-free, then $\mc{M}_0$ is a tilting bundle.
    \item Assumptions as in $(2)$. If $d$ is even, $\mc{M}_0$ is locally-free and $H^{d/2}(\mc{M}_0 \otimes \mc{M}_0^{\vee})=0$, then 
    $\mc{M}_0$ is a tilting bundle.
 \end{enumerate}
\end{thm}

\begin{proof}
$(1)$   Denote by $M:=\pi_*(\mc{E} \otimes \mc{E}^{\vee})$.
 By \cite[Proposition $1.2$]{yoshi}, it suffices to check that 
    \[H^i_x(M)=0\, \mbox{ for }\, i \not= \dim(X).\]
    Since $X$ is affine, $H^i(M)=0$ for $i \not= 0$. 
    As pushforward of torsion-free sheaves are torsion-free, $H^0_x(M)=0$. 
    By Proposition \ref{prop:loc}, it suffices to check that $H^i(U,M)=0$ for $i \not\in \{0,\dim(X)-1\}$.
    Since $X$ is Gorenstein and $\pi$ is a small resolution, $K_{\wt{X}} \cong \mo_X$. 
   Since $\mc{E}$ is a tilting bundle, Theorem \ref{thm:locdual}  implies that 
 \begin{equation}\label{eq:van2}
     H^{d-i}_E(\mc{E} \otimes \mc{E}^{\vee}) \cong H^i(\mc{E}^{\vee} \otimes \mc{E})^{\vee} = 0\, \mbox{ for all } 0<i<d.
 \end{equation}
 Using Proposition \ref{prop:loc}, we have exact sequences of the 
 \[H^i(\mc{E} \otimes \mc{E}^{\vee}) \to H^i(U,\mc{E} \otimes \mc{E}^{\vee}) \to H^{i+1}_E(\mc{E} \otimes \mc{E}^{\vee})\]
 Then, \eqref{eq:van2} implies that $H^i(U, \mc{E} \otimes \mc{E}^{\vee})=0$ for $1 \le i <d-1$.
This proves that $M$ is a maximal Cohen-Macaulay $\mo_X$-module. Moreover, since 
$\pi_*\mc{E}$ is a reflexive $\mo_X$-module, 
\[ \pi_*(\mc{E} \otimes \mc{E}^{\vee}) \cong \Hc_X(\pi_*\mc{E},\mc{E}). \]
Therefore, $\pi_*\mc{E}$ is a tilting factor of $M$. This proves $(1)$. 

We now prove $(2)$ and $(3)$.
Let $M$ be a  maximal Cohen-Macaulay $\mo_X$-module admitting a tilting factorization. Let $M_0$ be a tilting factor of $M$.
Denote by $\mc{M}_0:=(\pi^*M_0)^{\vee \vee}$ and $\mc{M}:=\mc{M}_0 \otimes \mc{M}_0^{\vee}$. 
By assumption, $\mc{M}_0$ is a locally-free sheaf. 
We need to show that $H^i(\mc{M})=0$  for all  $i>0$.
Note that, $\pi_*(\mc{M})$ is a reflexive $\mo_X$-module ($\Hc_{\wt{X}}(\mc{M}_0,\mc{M}_0)$ is reflexive and pushforward of reflexive sheaves under small resolutions are reflexive) which agrees with $M$ over $U$. Therefore, $M \cong \pi_*\mc{M}$.
Since $M$ is maximal Cohen-Macaulay and $X$ is affine,
\begin{equation}\label{eq:loc-2}
    H^i_x(M)=0\, \mbox{ for } i \not=d,\, \mbox{ which implies that } H^i(U,\mc{M})=H^{i}(U,M)=0\, \mbox{ for } 1 \le i <d-1.
\end{equation}
Let $j:U \to \wt{X}$ be the open immersion. Using the Grothendieck spectral sequence, we have 
\[ E_2^{k,l}= R^k \pi_*(R^l j_*(\mc{M}|_U)) \Rightarrow R^{k+l} (\pi \circ j)_* (\mc{M}|_U).\]
For $1 \le k+l \le d-2$, we have $R^{k+l} (\pi \circ j)_* (\mc{M}|_U) \cong H^{k+l+1}_x(M)=0$. 
Since $\mc{M}$ is locally-free, $E$ is local complete intersection and $\dim(E)=n$, then for  $l \ge 1$ 
 we have $R^l j_*(\mc{M}|_U)=\mc{H}^{l+1}_E(\mc{M})$, and for $l \not= d-n, \mc{H}^{l}_E(\mc{M})=0$ (see \cite[Proposition 2.6.8]{Dan2015}). 
In particular, for  $l \ge 1$ and $l \not= d-n-1$ 
\begin{equation}\label{eq:loc-1}
    R^l j_*(\mc{M}|_U)=\mc{H}^{l+1}_E(\mc{M})=0.
\end{equation}
Therefore, $E_2^{k,l}=0$ for $l \ge 1$, $l \not= d-n-1$ and all $k \in \mb{Z}$. This means $E_{\infty}^{k,l}=0$ for $l \ge 1$ and $l \not= d-n-1$.
Hence, for $0 \le k \le n-1$ we have 
\begin{equation}\label{eq:spec-3}
   E_\infty^{k,d-n-1}\, =\, R^{k+d-n-1} (\pi \circ j)_* (\mc{M}|_U)\, =\, 0. 
\end{equation}
We now calculate $E_{\infty}^{k,d-n-1}$ for all $0 \le k \le d-n-1$. Using the vanishings of $E_2^{k,l}$ above, we observe that 
\begin{equation}\label{eq:sepc-2}
    E_r^{k,d-n-1}=E_2^{k,d-n-1}\, \mbox{ if }\, r \le d-n\, \mbox{ and } E_\infty^{k,d-n-1}=E_{d-n+1}^{k,d-n-1}.
\end{equation}
For $r=d-n$, we have $E_{r+1}^{k,d-n-1}$ is the cohomology in the middle of the 
sequence:
\begin{equation}\label{eq:sepc-1}
    E_{d-n}^{k-d+n,2d-2n-2} \to E_{d-n}^{k,d-n-1} \to E_{d-n}^{k+d-n,0}.
\end{equation}
Since $d-n \ge 2$, we have $2d-2n-2 \ge 1$. This implies the first term in \eqref{eq:sepc-1} is $0$.
Moreover,  
\[E_2^{k+d-n,0}\, =\, R^{k+d-n}\pi_*(j_*\mc{M}|_U)\, =\, R^{k+d-n}\pi_*\mc{M}\, \mbox{ for } 0 \le k \le n-1, \]
which is zero if $d-n> \dim(E)=n$. This is the case if $n < d/2$. 
Moreover by assumption, if $d$ is even and $n=d/2$, then $H^{d/2}(\mc{M})=0$. 
Therefore, $E_2^{k+d-n,0}=0$ in the setup of $(2)$ and $(3)$. This implies the last term in \eqref{eq:sepc-1} is equal to $0$.
Hence, $E_{d-n+1}^{k,d-n-1}=E_2^{k,d-n-1}$. Combining \eqref{eq:sepc-2} with \eqref{eq:spec-3}, we conclude that 
for $0 \le k \le n-1$,  
\[H^k(\mc{H}_E^{d-n}(\mc{M}))\, =\, H^k(R^{d-n-1} j_* \mc{M}|_U)\, =\, E_2^{k,d-n-1}\, =\, E_{\infty}^{k,d-n-1}= 0.\]
 Using the vanishings \eqref{eq:loc-1} along with the spectral sequence  
\[E_2^{k,l}= H^k(\mc{H}^l_E(\mc{M})) \Rightarrow H^{k+l}_E(\mc{M}),\, \mbox{ we conclude }
H^k(\mc{H}^{d-n}_E(\mc{M}))=H^{k+d-n}_E(\mc{M}).\]
This implies 
$H^i_E(\mc{M})=0$  for $i \le d-1$. Combining with \eqref{eq:loc-2}, the local cohomology long exact sequence 
(see Proposition \ref{prop:loc}) implies that $H^i(\mc{M})=0$ for $1 \le i \le d-2$. Moreover, since the fibers of $\pi$ are of dimension less
than or equal to $d-2$, we have $H^i(\mc{M})=0$ for $i>d-2$. This proves $H^i(\mc{M})=0$  for all  $i>0$ i.e., $\mc{M}_0$ is a tilting bundle.
This proves the theorem.
\end{proof}

\begin{cor}\label{cor:tb-2}
 Suppose that $X$ is Gorenstein and  $\mo_{\wt{X}} \oplus \left(\oplus_{i=1}^m \mc{L}_i\right)$ is a tilting bundle for some invertible sheaves $\mc{L}_i$. Then, 
    \[ \pi_*(\mc{L}_i),\, \pi_*(\mc{L}_i^{\vee}),\, \pi_*(\mc{L}_i^{\vee} \otimes \mc{L}_j)\, \mbox{ for } 1 \le i \le m\, \mbox{ and }\, 1 \le j \le m\]
    are all maximal Cohen-Macaulay $\mo_X$-module.
\end{cor}

\begin{proof}
 Denote by $\mc{E}\, :=\oplus_{i=1}^m \mo_X \oplus \mc{L}_i$. Then 
\[\pi_*(\mc{E} \otimes \mc{E}^{\vee}) \cong \mo_X^{\oplus m+1} \oplus \left(\bigoplus_i^m \pi_*\mc{L}_i\right) \oplus \left(\bigoplus_i^m \pi_*\mc{L}_i^{\vee}\right) \oplus \bigoplus\limits_{i \not= j} \pi_*(\mc{L}_i^{\vee} \otimes \mc{L}_j) \]
By Theorem \ref{thm:tf}, $\pi_*(\mc{E} \otimes \mc{E}^{\vee})$ is a maximal Cohen-Macaulay $\mo_X$-module. 
This implies 
\[H^j_x(\pi_*(\mc{E} \otimes \mc{E}^{\vee}))=0\, \mbox{ for }\, j \not= d.\]
Since $H^j_x(-)$ commutes with direct sum. This implies 
\[H^j_x(\mo_X)=0, \, H^j_x(\pi_*\mc{L}_i)=0, \, H^j_x(\pi_*\mc{L}_i^{\vee})=0\, \mbox{ and } H^j_x(\pi_*(\mc{L}_i^{\vee} \otimes \mc{L}_k))=0\, \mbox{ for all }
j \not= d \mbox{ and }\]
$1 \le i, k \le m$.
This implies $\pi_*(\mc{L}_i)$, $\pi_*(\mc{L}_i^{\vee})$ and  $\pi_*(\mc{L}_i^{\vee} \otimes \mc{L}_k)$
are all maximal Cohen-Macaulay $\mo_X$-module
for $1 \le i,k \le m$. This proves the corollary.     
\end{proof}

\section{Properties of almost full sheaves}
Throughout this section we use the setup of \S \ref{se01} above. 
In particular, $X$ is not necessarily Gorenstein. Let $d=\dim(X)$.  
Consider a small resolution of $X$:
\[\pi:\wt{X} \to X.\]  
Denote by $E$ the exceptional locus (which is not necessarily a divisor) and $U:=X \backslash \{x\}$. 

\begin{defi}
    A locally-free sheaf $\mc{M}$ of $\wt{X}$ is said to \emph{almost full} if
$\pi_*\mc{M}$ is a maximal Cohen-Macaulay $\mo_X$-module.
\end{defi}

\begin{exa}\label{exa:af}
    Let $(X,x)$ be any $3$-dimensional isolated Cohen-Macaulay singularity and $\pi: \wt{X} \to X$ be a resolution (not necessarily small).
    Let $D \subset X$ be a normal, integral surface and $\wt{D} \subset \wt{X}$ be the strict transform of $D$. Then, $\mo_{\wt{X}}(-\wt{D})$
    is almost full. Indeed, apply $\pi_*$ to the short exact sequence:
    \[0 \to \mo_{\wt{X}}(-\wt{D}) \to \mo_{\wt{X}} \to \mo_{\wt{D}} \to 0\]
    and use depth comparison in short exact sequence (normal surface are Cohen-Macaulay).
\end{exa}

\begin{exa}\label{example:ab}
Take $(X,x)$ to be an $n$-dimensional isolated Cohen-Macaulay singularity and let $\pi:\wt{X}\rightarrow X$ be a resolution (not necessarily crepant).
Let $\mc{F}$ be a coherent sheaf on $\wt{X}$ such that 
\begin{enumerate}
    \item $\Ext^i(\mc{F},K_{\wt{X}}) = 0$ for $i>0$, and
    \item $H^i(\wt{X}, \mc{F}) = 0$ for $i>0$.
\end{enumerate}
    Then the pushforward $\pi_*\mc{F}$ is a Cohen-Macaulay module i.e., $\mc{F}$ is an almost full sheaf. 
    The proof follows identically as in \cite[Lemma 3.1]{ha17}, after using $\pi^!\omega_X \cong K_{\wt{X}}$.
\end{exa}

\subsection{Criterion for almost full sheaves}
Here we give various equivalent conditions for a locally-free sheaf on $\wt{X}$ to be almost full.

\begin{lem}\label{lem:mck01}
 Let $\mc{M}$ is a locally-free sheaf on $\wt{X}$. Then, $\mc{M}$ is almost full if and only if 
 the following cohomological condition holds:
 \[H^i_E(\mc{M}) \cong H^i(\mc{M})\, \mbox{ for all } 1 \le i \le d-2\, \mbox{ and }\, H^{d-1}_E(\mc{M}) \to H^{d-1}(\mc{M})\, 
 \mbox{ is injective}. \]
\end{lem}

\begin{proof}
    Note that, $M:=\pi_*\mc{M}$ is a maximal Cohen-Macaulay $\mo_X$-module if and only if \[H^i_x(M)=0,\,  \mbox{ for all }\, i \not= d.\]
    Since $X$ is local, the local cohomology exact sequence (see Proposition \ref{prop:loc}) implies that 
    \[H^i(U,M)=0,\, \mbox{ for all }\, i \not\in \{0,d-1\}\, \mbox{ and }\, H^0(M)=H^0(U,M).\]
    Using the local cohomology exact sequence once more, this is equivalent to 
    \[H^i_E(\mc{M}) \to H^i(\mc{M})\, \mbox{ is an isomorphism for } 1 \le i \le d-2\, \mbox{ and } H^{d-1}_E(\mc{M}) \to H^{d-1}(\mc{M})\, 
 \mbox{ is injective}. \]
 This proves the lemma.
\end{proof}

\begin{lem}\label{lem:nor}
 Let $\mc{L}$ be an invertible sheaf such that there exists a non-zero section $s \in H^0(\mc{L})$
 with irreducible, non-singular zero locus $\wt{D}:=Z(s)$. Suppose $d \ge 3$ and $D:=\pi(\wt{D})$. Then, 
 \begin{enumerate}
     \item  $\mc{L}^{\vee}$ is almost full if and only if $\mo_D$ is a Cohen-Macaulay $\mo_X$-module. In particular, if $\mc{L}^{\vee}$
     is almost full, then  $D$ is normal.
 \item Suppose that $X$ is Gorenstein. Then, $\mc{L}$ is almost full if and only if $\mo_D$ is a Cohen-Macaulay $\mo_X$-module. In particular, 
 if $\mc{L}$ is almost full, then $D$ is normal.
 \end{enumerate}
\end{lem}

\begin{proof}
Consider the short exact sequences:
\begin{align*}
  &  0 \to \mo_{\wt{X}} \xrightarrow{s} \mc{L} \to \N_{\wt{D}|\wt{X}} \to 0 \\
  & 0 \to \mc{L}^{\vee} \xrightarrow{s} \mo_{\wt{X}} \to \mo_{\wt{D}} \to 0
\end{align*}
Applying $\pi_*$, we get the short exact sequences:
\begin{align}
  &  0 \to \mo_{X} \to \pi_*\mc{L} \to \pi_*\N_{\wt{D}|\wt{X}} \to 0 \label{eq:nor-1}\\
  & 0 \to \pi_*\mc{L}^{\vee} \to \mo_X \to \mo_{D} \to 0 \label{eq:nor-2}
\end{align}
The exactness of \eqref{eq:nor-1} follows from $R^1\pi_*\mo_{\wt{X}}=0$ (Remark \ref{rem:crep}).
The exactness of \eqref{eq:nor-2} follows from the fact that the morphism from $\mo_X$ to $\pi_*\mo_{\wt{D}}$ factors
through $\mo_D$. 

 $(1)$ Using depth comparison on \eqref{eq:nor-2}, we conclude that $\mo_D$ is Cohen-Macaulay if and only if $\mc{L}^{\vee}$ is almost full.
 In particular, if $\mc{L}^{\vee}$ is almost full then $D$ satisfies Serre's criterion $S_2$ and $R_1$ (as 
 $D$ has isolated singularity). Therefore, $D$ is normal.

 $(2)$ Suppose that $X$ is Gorenstein. 
 Dualizing \eqref{eq:nor-1} gives us:
 \begin{equation}\label{eq:dual}
     0 \to (\pi_*\mc{L})^{\vee} \to \mo_X \to \mr{Ext}^1_X(\pi_*\N_{\wt{D}|\wt{X}}, \mo_X) \xrightarrow{\phi} \mr{Ext}^1_X(\pi_*\mc{L},\mo_X)
 \end{equation}
 Note that, $\mr{Ext}^1_X(\pi_*\N_{\wt{D}|\wt{X}}, \omega_X)$ is of rank one as an $\mo_D$-module and $\mr{Ext}^1_X(\pi_*\mc{L},\mo_X)$ is supported on $x$. 
 Denote by $\mc{A}$ the kernel of the morphism $\phi$. Since there is a 
 surjective morphism from $\mo_X$ to $\mc{A}$, this implies $\mc{A} \cong \mo_D$.
 Then, $\mc{L}$ is almost full if and only if $\pi_*\mc{L}$ is maximal Cohen-Macaulay if and only if 
 $(\pi_*\mc{L})^{\vee}$ is maximal Cohen-Macaualay. 
 Using depth comparison in \eqref{eq:dual}, $(\pi_*\mc{L})^{\vee}$ is maximal Cohen-Macaulay
 if and only if $\mo_D$ is Cohen-Macaulay. The rest of $(2)$ follows identically as $(1)$. This proves the lemma.
 \end{proof}



 \subsection{Rank one almost full sheaves}\label{sec:alm}
Let $\mc{L}$ be an invertible sheaf on $\wt{X}$.
Suppose there exists a section $s \in H^0(\mc{L})$ such that the zero locus $\wt{D}:=Z(s)$ of the section $s$ is irreducible and  
non-singular.
Denote by $D:=\pi(\wt{D})$ the (reduced) scheme-theoretic image of $\wt{D}$.

\begin{thm}\label{thm:mck}
The invertible sheaf $\mc{L}^{\vee}$ is almost full if and only if \[H^i(\mo_{\wt{D}})=0,\, \mbox{ for all }\, 1 \le i <d-1\, \mbox{ and }
H^0(\mo_{\wt{X}}) \to H^0(\mo_{\wt{D}})\, \mbox{ is surjective}.\]
\end{thm}

\begin{proof}
      Combining Lemma \ref{lem:mck01} and Theorem \ref{thm:locdual}, $\mc{L}^{\vee}$ is almost full
    if and only if  the following cohomological conditions are satisfied:
     \[H^i_E(\mc{L} \otimes K_{\wt{X}}) \cong H^i(\mc{L} \otimes K_{\wt{X}})\, \mbox{ for } 2 \le i \le d-1\, \mbox{ and } H^1_E(\mc{L} \otimes K_{\wt{X}}) \to H^1(\mc{L} \otimes K_{\wt{X}})\, \mbox{ is surjective}.\]     
  Consider the short exact sequence:
  \begin{equation}\label{eq:one-1}
      0 \to \mo_{\wt{X}} \to \mc{L} \to \N_{\wt{D}|\wt{X}} \to 0
  \end{equation}
  Recall, $K_{\wt{D}} \cong K_{\wt{X}} \otimes \N_{\wt{D}|\wt{X}}$. 
  Tensoring \eqref{eq:one-1} by $K_{\wt{X}}$  and considering the associated cohomology long exact 
     we get the following commutative diagram of exact sequences:
     \begin{equation}\label{eq05}
         \begin{diagram}
        H^{i}_E(K_{\wt{X}})&\rTo&H^i_E(\mc{L} \otimes K_{\wt{X}})&\rTo&H^i_E(K_{\wt{D}})&\rTo&H^{i+1}_E(K_{\wt{X}})\\
         \dTo&\circlearrowleft&\dTo&\circlearrowleft&\dTo&\circlearrowleft&\dTo\\
         H^{i}(K_{\wt{X}})&\rTo&H^i(\mc{L} \otimes K_{\wt{X}})&\rTo&H^i(K_{\wt{D}})&\rTo&H^{i+1}(K_{\wt{X}})
     \end{diagram}
     \end{equation}
  By Grauert-Riemenschneider vanishing theorem, we have:
  \[H^i(K_{\wt{X}})=0=H^i(K_{\wt{D}})\, \mbox{ for } i \ge 1, \mbox{ and } H^i_E(K_{\wt{X}})=H^{d-i}(\mo_{\wt{X}})=0\, \mbox{ for } i \not\in \{0,d\}.\]
  Therefore, $H^i_E(\mc{L} \otimes K_{\wt{X}})\cong H^{i}_E(K_{\wt{D}})$ for $1 \le i < d-1$ and $H^i(\mc{L} \otimes K_{\wt{X}})=0$ for all $i>0$.
  This means that $\mc{L}^{\vee}$ is almost full if and only if 
  \[0= H^i(\mc{L} \otimes K_{\wt{X}}) \cong H^i_E(\mc{L} \otimes K_{\wt{X}}) \cong H^i_E(K_{\wt{D}})=H^{d-i-1}(\mo_{\wt{D}})\, \mbox{ for } 1 \le i < d-1\, \mbox{ and } \]\[
  H^{d-1}_E(\mc{L} \otimes K_{\wt{X}}) \cong H^{d-1}(\mc{L} \otimes K_{\wt{X}}) = 0.\]
  By Serre duality, $H^{d-1}_E(\mc{L} \otimes K_{\wt{X}})=H^1(\mc{L}^{\vee})$.
  Since $H^1(\mo_{\wt{X}})=0$, the following short exact sequence implies that $H^1(\mc{L}^{\vee})=0$ if and only if the restriction map from 
  $H^0(\mo_{\wt{X}})$ to $H^0(\mo_{\wt{D}})$ is surjective:
  \begin{equation}\label{eq:zero}
     0 \to \mc{L}^{\vee} \to \mo_{\wt{X}} \to \mo_{\wt{D}} \to 0.
 \end{equation}
This proves the theorem.
\end{proof}

\begin{rem}\label{rem:pd}
 Let $\mc{F}$ be a coherent $\mo_{\wt{X}}$-module. Dual of coherent sheaves are reflexive i.e., $\mc{F}^{\vee}$ and $(\pi_*\mc{F})^{\vee}$ are reflexive
 modules on $\wt{X}$ and $X$, respectively. 
 Pushforward of reflexive sheaves under small resolutions are reflexive i.e., $\pi_*(\mc{F}^{\vee})$ is a reflexive $\mo_X$-module.
 Since $(\pi_*\mc{F})^{\vee}$ and $\pi_*(\mc{F}^{\vee})$ agree over the open set $U$ and $\mr{codim}(X\backslash U) \ge 2$, this means 
 $(\pi_*\mc{F})^{\vee} \cong \pi_*(\mc{F}^{\vee})$.
\end{rem}

\begin{lem}\label{lem:dual}
Suppose that $X$ is Gorenstein. 
    If $\mc{L}$ is almost full, then so is $\mc{L}^{\vee}$.
    Conversely, if $\mc{L}^{\vee}$ is almost full, then $\mc{L}$ and $\mc{L} \otimes \mo_{\wt{D}}$ are both almost full.
\end{lem}

\begin{proof}
 Suppose that $\mc{L}$ is almost full. This implies $\pi_*\mc{L}$ is maximal Cohen-Macaulay. Since $X$ is Gorenstein, 
 $(\pi_*\mc{L})^{\vee}$ is also maximal Cohen-Macaulay. Using Remark \ref{rem:pd}, we conclude 
 that $(\pi_*\mc{L})^{\vee} \cong \pi_*(\mc{L}^{\vee})$ i.e., $\mc{L}^{\vee}$ is almost full. This proves the first part of the lemma.

   Suppose now that $\mc{L}^{\vee}$ is almost full.  Arguing as above, we have $\mc{L}$ is almost full. We now show that 
   $\mc{L} \otimes \mo_{\wt{D}}$ is almost full as an $\mo_{\wt{D}}$-module.  
   Note that, $\N_{\wt{D}|\wt{X}} \cong \mc{L} \otimes \mo_{\wt{D}}$.
   Since $\mc{L}$ is almost full, $\pi_*\mc{L}$ is maximal Cohen-Macaulay. Applying depth comparison to 
   \eqref{eq:nor-1} we conclude that $\pi_*\N_{\wt{D}|\wt{X}} \cong \pi_*\left(\mc{L} \otimes \mo_{\wt{D}}\right)$
   is Cohen-Macaulay. Therefore, $\mc{L} \otimes \mo_{\wt{D}}$ is almost full as an $\mo_{\wt{D}}$-module.
   Thie proves the lemma.
\end{proof}

\subsection{Tensor product of almost full sheaves}
  Let $\mc{L}_1$ and $\mc{L}_2$ be two invertible sheaves on $\wt{X}$ such that the dual sheaves 
  $\mc{L}_1^{\vee}$ and $\mc{L}_2^{\vee}$ are almost full. Suppose there exist 
    sections $s_1 \in H^0(\mc{L}_1)$ and $s_2 \in H^0(\mc{L}_2)$ such that the zero loci
    $Z(s_1)$, $Z(s_2)$ and $Z(s_1).Z(s_2)$ are non-singular and intersects the exceptional locus $E$ properly.
    Assume further that $Z(s_1)$ intersects $Z(s_2)$ properly.
    Using Lemma \ref{lem:nor}, the image under $\pi$ of the zero loci $\pi(Z(s_1))$ and $\pi(Z(s_2))$ are Cohen-Macaulay schemes
    with isolated singularities. In particular, the restriction of $\pi$ to these zero loci are small resolutions of 
    $\pi(Z(s_1))$ and $\pi(Z(s_2))$.
 Under these assumptions we prove:
    
\begin{thm}\label{thm:tens}
Let $\mc{L}_1, \mc{L}_2$ and $s_2 \in H^0(\mc{L}_2)$ be as above.  
Assume further that a general element of the linear system $|\mc{L}_1 \otimes \mc{L}_2|$ is non-singular intersecting 
 $E$ properly.
  If $\mc{L}_1^{\vee}|_{Z(s_2)}$ is almost full as an $\mo_{Z(s_2)}$-module, then $\mc{L}_1^{\vee} \otimes \mc{L}_2^{\vee}$ is almost full.    
\end{thm}

\begin{proof}
    For this we use the criterion given in Theorem \ref{thm:mck}. Denote by 
    \[\wt{D}_1:=Z(s_1),\, \wt{D}_2:= Z(s_2),\, D_1:=\pi(\wt{D}_1)\, \mbox{ and } D_2:=\pi(\wt{D}_2).\]
    Consider the short exact sequences: 
    \begin{align}
      &   0 \to (\mc{L}_1 \otimes \mc{L}_2)^{\vee} \to \mo_{\wt{X}} \to \mo_{\wt{D}_1 \cup \wt{D}_2} \to 0 \label{eq:tens-4}\\
     &   0 \to \mc{L}_1^{\vee} \to \mo_{\wt{X}} \to \mo_{\wt{D}_1}  \to 0 \label{eq:tens-1}
    \end{align}
       and the restriction to $\wt{D}_2$ (exactness follows because $\wt{D}_1$ intersects $\wt{D}_2$ properly):
    \[0 \to \mc{L}_1^{\vee}|_{\wt{D}_2} \to \mo_{\wt{D}_2} \to \mo_{\wt{D}_1.\wt{D}_2} \to 0 \]
    We first claim that $H^1((\mc{L}_1 \otimes \mc{L}_2)^{\vee})=0$. Indeed, since $H^1(\mo_{\wt{X}})=0$, it suffices to 
    show that the induced map from $H^0(\mo_{\wt{X}})$ to $H^0(\mo_{\wt{D}_1 \cup \wt{D}_2})$ is surjective.
    Consider the short exact sequence:
    \begin{equation}\label{eq:div}
        0 \to \mo_{\wt{D}_1 \cup \wt{D}_2} \to \mo_{\wt{D}_1} \oplus \mo_{\wt{D}_2} \to \mo_{\wt{D}_1.\wt{D}_2} \to 0.
    \end{equation}
    Since $\mc{L}_1^{\vee}, \mc{L}_2^{\vee}$ and $\mc{L}_1^{\vee}|_{\wt{D}_2}$ are almost full, we have (due to Zariski's main theorem, after observing that 
    $D_1, D_2, D_1.D_2$ are normal by Lemma \ref{lem:nor}):
    \[H^0(\mo_{\wt{D}_1})=\mo_{D_1},\, H^0(\mo_{\wt{D}_2})=\mo_{D_2}\, \mbox{ and } H^0(\mo_{\wt{D}_1.\wt{D}_2})=\mo_{D_1.D_2} \]
    Applying the $H^0(-)$ to \eqref{eq:div}, we get the short exact sequence (as $D_1.D_2 \subset D_1$ and $D_2$):
    \[0 \to H^0(\mo_{\wt{D}_1 \cup \wt{D}_2}) \to \mo_{D_1} \oplus \mo_{D_2} \to \mo_{D_1.D_2} \to 0.\]
    Using the universal property of kernel, this implies 
    \[H^0(\mo_{\wt{D}_1 \cup \wt{D}_2}) \cong \mo_{D_1 \cup D_2}.\]
    As $D_1 \cup D_2 \subset X$, this means that the morphism from $H^0(\mo_{\wt{X}})=\mo_X$ to $H^0(\mo_{\wt{D}_1 \cup \wt{D}_2}) \cong \mo_{D_1 \cup D_2}$
    is surjective. This proves our claim that $H^1((\mc{L}_1 \otimes \mc{L}_2)^{\vee})=0$.
    
    We now claim that for a general section $s_3 \in H^0(\mc{L}_1 \otimes \mc{L}_2)$, we have 
    \begin{equation}\label{eq:tens-3}
        h^i(\mo_{Z(s_3)})=0\, \mbox{ for } 1 \le i <d-1\, \mbox{ and } H^0(\mo_{\wt{X}}) \to H^0(\mo_{Z(s_3)})\, \mbox{ is surjective.}
    \end{equation}
    The surjectivity condition follows immediately from the above claim that $H^1((\mc{L}_1 \otimes \mc{L}_2)^{\vee})=0$.
    For the vanishing cohomologies of $Z(s_3)$ consider the short exact sequence \eqref{eq:div}. The associated long exact sequence gives us:
    \begin{equation}\label{eq:tens-2}
        ... \to H^{i-1}(\mo_{\wt{D}_1.\wt{D}_2}) \to H^i(\mo_{\wt{D}_1 \cup \wt{D}_2}) \to H^i(\mo_{\wt{D}_1}) \oplus H^i(\mo_{\wt{D}_2}) \to ...
    \end{equation}
     Since $\mc{L}_1^{\vee}, \mc{L}_2^{\vee}$ and $\mc{L}_1^{\vee}|_{\wt{D}_2}$ are almost full, Theorem \ref{thm:mck} implies 
    \[h^i(\mo_{\wt{D}_1})=0=h^i(\mo_{\wt{D}_2})\, \mbox{ for } 1 \le i < d-1 \, \mbox{ and } h^i(\mo_{\wt{D}_1.\wt{D}_2})=0\, \mbox{ for } 1 \le i <d-2.\]
 Applying this to \eqref{eq:tens-2} and using $H^0(\mo_{\wt{D}_1}) \oplus H^0(\mo_{\wt{D}_2}) \twoheadrightarrow H^0(\mo_{\wt{D}_1.\wt{D}_2})$,
 we conclude that $h^i(\mo_{\wt{D}_1 \cup \wt{D}_2})=0$ for $1 \le i <d-1$.
  Using \eqref{eq:tens-4} this implies $H^i((\mc{L}_1 \otimes \mc{L}_2)^{\vee})=0$ for 
  $1 \le i \le d-1$. 
  For a general $s_3 \in H^0(\mc{L}_1 \otimes \mc{L}_2)$ as above, we get the short exact sequence \eqref{eq:tens-4} as above 
  with $\wt{D}_1 \cup \wt{D}_2$ replaced by the zero locus $Z(s_3)$. Then, the vanishing of the cohomologies of 
  $(\mc{L}_1 \otimes \mc{L}_2)^{\vee}$ implies \eqref{eq:tens-3}.
    This proves the claim. Finally, Theorem \ref{thm:mck} implies that 
    $(\mc{L}_1 \otimes \mc{L}_2)^{\vee}$ is almost full. This proves the theorem.
\end{proof}

\subsection{Vanishing result}
We now show vanishing cohomologies of almost full sheaves.

\begin{prop}\label{prop:van}
Let $\mc{L}$ be an invertible sheaf such that there exists a section 
$s \in H^0(\mc{L})$ whose zero locus is non-singular and intersects $E$ properly.
    If $\mc{L}^{\vee}$ is almost full, then $H^i(\mc{L}^{\vee})=0$ for all $i>0$.
\end{prop}

\begin{proof}
Let $s \in H^0(\mc{L})$ be a section such that the zero locus $\wt{D}:=Z(s)$
is non-singular and intersects $E$ properly. Denote by $D:=\pi(\wt{D})$ the image of $\wt{D}$.
 Therefore, we have a short exact sequence of the form 
\begin{equation}\label{eq:zero-1}
0 \to \mc{L}^{\vee} \to \mo_{\wt{X}} \to \mo_{\wt{D}} \to 0
\end{equation}
Taking the associated cohomology long exact sequence, 
\[ ... \to H^0(\mo_{\wt{X}}) \to H^0(\mo_{\wt{D}}) \to H^1(\mc{L}^{\vee}) \to H^1(\mo_{\wt{X}})\, \mbox{ and for all } i>1,  H^{i-1}(\mo_{\wt{D}}) \to H^i(\mc{L}^{\vee}) \to H^i(\mo_{\wt{X}})\]
are exact. Remark \ref{rem:crep} implies that $H^i(\mo_{\wt{X}})=0$ for $i>0$.
Since $\mc{L}^{\vee}$ is almost full, Theorem \ref{thm:mck} implies that $H^i(\mo_{\wt{D}})=0$ for all $1 \le i <d-1$. 
Since $\wt{D}$ intersects $E$ properly, 
\[\pi_D: \wt{D} \to D\]
is a small resolution. This implies $H^{i}(\mo_{\wt{D}})=0$ for all $i \ge d-1$.

By Lemma \ref{lem:nor}, $D$ is a normal, Cohen-Macaulay variety. 
Since $D$ is normal, Zariski's main theorem implies that $H^0(\wt{D})=\mo_D$.
Therefore, the map from $H^0(\mo_{\wt{X}})$ to $H^0(\mo_{\wt{D}})$ is surjective. Using the long exact sequences above, this implies $H^i(\mc{L}^{\vee})=0$ for all $i>0$.
This proves the proposition.
\end{proof}

We now observe that by using this proposition repeatedly, we can prove similar vanishing result on higher tensor powers of $\mc{L}$.

\begin{defi}\label{defi:cm-deg}
    An invertible sheaf $\mc{L}$ on $\wt{X}$ is said to have \emph{CM-degree} $m$, if $m \le \dim(E)$ and 
    there exist $m$ global sections $s_1,...,s_m \in H^0(\mc{L})$
    such that  for all  $1 \le j \le m$, 
    \[\wt{D}(j):=Z(s_1) \cap Z(s_2) \cap ... \cap Z(s_j), \, \] is  irreducible, non-singular, of codimension $j$,  intersects $E$ properly and 
    the (reduced) image $D(j):=\pi(\wt{D}(j))$ 
    is a Cohen-Macaulay scheme.  Here $Z(s_j)$ denotes the zero locus of $s_j$. Note that, if $\mc{L}$ is of CM-degree $m$, then it is of 
    CM-degree $l$ for all $0 \le l \le m$.
\end{defi}

\begin{exa}
    In the Example \ref{exa:af}, $\mo_{\wt{X}}(\wt{D})$ has CM-degree equal to the dimension of the exceptional locus.
\end{exa}

\begin{thm}\label{thm:van-main}
    Let $\mc{L}$ be an invertible sheaf and $n:=\dim(E)$. Suppose that $\mc{L}$ has CM-degree $m$ for some $m \le n$.  
    Then, 
     $(\mc{L}^\vee)^{\otimes j}$ is almost full for all $1 \le j \le m$. In particular, 
    $H^i((\mc{L}^{\vee})^{\otimes j})=0$ for all $i>0$ and $1 \le j \le m$.   
\end{thm}

\begin{proof}
Since $\mc{L}$ has CM-degree $m$, there exists $m \le n$ global sections $s_1,...,s_m \in H^0(\mc{L})$
    such that for all $1 \le j \le m$, 
    \[\wt{D}(j):=Z(s_1) \cap Z(s_2) \cap ... \cap Z(s_j)\,\]  is  irreducible, non-singular of codimension $j$, intersects  $E$  properly
   and $D(j):=\pi(\wt{D}(j))$ 
    is a Cohen-Macaulay scheme. 
For each $1 \le j \le m$, denote by $\mc{L}_j:=\mc{L}|_{\wt{D}(j)}$ the restriction of $\mc{L}$ to $\wt{D}(j)$.
We have a short exact sequence of the form:
\begin{equation}\label{eq:recur-1}
    0 \to \mc{L}_j^{\vee} \xrightarrow{.s_{j+1}} \mo_{\wt{D}(j)} \to \mo_{\wt{D}(j+1)} \to 0
\end{equation}
Since $D(j)$ is a Cohen-Macaulay scheme of dimension at least $2$, with isolated singularities, Serre's criterion implies that $D(j)$ is normal.
Then, Zariski's main theorem implies  that 
$\pi_*\mo_{\wt{D}(j)}=\mo_{D(j)}$. Since $D(j+1) \subset D(j)$, we get the following short exact sequence  by applying $\pi_*$ to \eqref{eq:recur-1} 
\[0 \to \pi_*\left(\mc{L}_j^{\vee}\right) \to \mo_{D(j)} \to \mo_{D(j+1)} \to 0.\]
By depth comparison in exact sequence, $\pi_*(\mc{L}_j^{\vee})$ is a Cohen-Macaulay $\mo_{D(j)}$-module i.e., $\mc{L}_j^{\vee}$ is almost 
full as an $\mo_{\wt{D}(j)}$-module. 
We now show that $(\mc{L}^{\vee})^{\otimes j}$ is almost full for all $1 \le j \le m$.
Indeed, since $\mc{L}_{j-2}^{\vee}$ and $\mc{L}^{\vee}_{j-1}=\mc{L}^{\vee}_{j-2}|_{Z(s_{j-1})}$ are almost full, Theorem \ref{thm:tens} implies that 
$(\mc{L}_{j-2}^{\vee})^{\otimes 2}$ is almost full.
Recursively, if $(\mc{L}_{j-i}^{\vee})^{\otimes i}=(\mc{L}_{j-i-1}^{\vee})^{\otimes i}|_{Z(s_{j-i})}$ is almost full, then 
Theorem \ref{thm:tens} implies $(\mc{L}_{j-i-1}^\vee)^{\otimes i+1}$ is almost full. Varying $i$ from $1$ to $j$, we conclude that 
$(\mc{L}^{\vee})^{\otimes j}$ is almost full. This proves the claim and hence the first part of the theorem.
   The second part of the theorem follows immediately from Proposition \ref{prop:van}. This proves the theorem.
\end{proof}

\begin{cor}\label{cor:tb}
    Let $\mc{L}$ be an invertible sheaf of CM-degree $m$. Then, 
    \[\mc{E}_l:=\mo_{\wt{X}} \oplus \bigoplus\limits_{i=1}^l \mc{L}^{\otimes i}\] is a tilting 
    bundle for all $1 \le l \le m$. Moreover, if $X$ is Gorenstein, then $\bigoplus\limits_{i=-m}^m (\pi_*\mc{L}^{\otimes i})$ is a maximal Cohen-Macaulay $\mo_X$-module 
    admitting a tilting factorization with tilting factor $\pi_*\mc{E}_m$.
\end{cor}

\begin{proof}
 Since $\mc{L}$ is of CM-degree $m$,  Theorem \ref{thm:van-main} implies that $(\mc{L}^{\vee})^{\otimes j}$ is almost full  and 
 \[H^i((\mc{L}^{\vee})^{\otimes j})=0\, \mbox{ for all }\, i>0\, \mbox{ and } 1 \le j \le m.\] 
    Then, \cite[Lemma $2.1$]{toda2} implies that 
   $H^i(\mc{L}^{\otimes j})=0$ for all  $i>0$, $1 \le j \le m$.
 This implies 
 \[H^i(\Hc_{\wt{X}}(\mc{E}_l,\mc{E}_l))=\bigoplus_{j=-l}^{l} H^i(\mc{L}^{\otimes j})=0\, \mbox{ for }\, i>0.\]
 Therefore, $\mc{E}_l$ is a tilting bundle for all $1 \le l \le m$. This proves the first part of the corollary.

 Suppose now that $X$ is Gorenstein. Theorem \ref{thm:van-main} along with Lemma \ref{lem:dual} also implies that $\pi_*\mc{L}^{\otimes i}$ is maximal Cohen-Macaulay 
 for all $-m \le i \le m$. Therefore,  $\bigoplus\limits_{i=-m}^m (\pi_*\mc{L}^{\otimes i})$ is a maximal Cohen-Macaulay $\mo_X$-module and  
 \[ \bigoplus\limits_{i=-m}^m (\pi_*\mc{L}^{\otimes i}) \cong \mr{Hom}_X(\pi_*\mc{E}_m,\pi_*\mc{E}_m). \]
 Hence,  $\bigoplus\limits_{i=-m}^m (\pi_*\mc{L}^{\otimes i})$ admits a tilting factorization with tilting factor $\pi_*\mc{E}_m$.
 This proves the corollary.
\end{proof}

Observe that the first part of Corollary \ref{cor:tb} does not require $X$ to be Gorenstein. 
We now prove the converse of Corollary \ref{cor:tb} in the Gorenstein setup.

\begin{cor}\label{lem:tilt-2}
Suppose that $X$ is Gorenstein. 
Let $\mc{L}$ be a very ample invertible $\mo_{\wt{X}}$-module and $n:=\dim(E)$.
 If    \[\mc{E}:= \mo_{\wt{X}} \oplus  \bigoplus\limits_{i=1}^{n} \mc{L}^{\otimes j}\] is a tilting bundle, then 
 $\mc{L}$ is almost full of CM-degree $n$.
\end{cor}
    
    \begin{proof}
        Since $\mc{L}$ is very ample, there exist $n$ global sections $s_1, s_2, ..., s_n$ such that 
        \[\wt{D}(j):=Z(s_1) \cap Z(s_2) \cap ... \cap Z(s_j), \, \mbox{ for all } 1 \le j \le n\] is  irreducible, non-singular, of codimension $j$,  intersects $E$ properly. It remains to show that $D(j):=\pi(\wt{D}(j))$ is a Cohen-Macaulay scheme. 
        By Corollary \ref{cor:tb-2}, $(\mc{L}^\vee)^{\otimes j}$ is almost full for all $1 \le j \le n$. 
        Denote by $\mc{L}_j:=\mc{L}|_{\wt{D}(j)}$. Then, we have the short exact sequence \eqref{eq:recur-1}.
        This gives rise to the following short exact sequences for all $1 \le j \le n$:
        \begin{align}
            & 0 \to (\mc{L}^{\vee})^{\otimes j} \to (\mc{L}^{\vee})^{\otimes j-1} \to (\mc{L}_1^{\vee})^{\otimes j-1} \to 0\label{eq:recur-3} \\
            & 0 \to (\mc{L}_i^{\vee})^{\otimes j-i} \to (\mc{L}_i^{\vee})^{\otimes j-i-1} \to (\mc{L}^{\vee}_{i+1})^{\otimes j-i-1} \to 0\, \, \mbox{ for all } 0 \le i \le j \le n. \label{eq:recur-4}
        \end{align}
        Proposition \ref{prop:van} implies that we have a short exact sequence (using the vanishing of $H^1)$:
        \[0 \to \pi_*((\mc{L}^{\vee})^{\otimes j}) \to \pi_*((\mc{L}^{\vee})^{\otimes j-1}) \to \pi_*((\mc{L}_1^{\vee})^{\otimes j-1}) \to 0.\]
        By depth comparison in short exact sequences, this implies $\pi_*((\mc{L}_1^{\vee})^{\otimes j-1})$ is Cohen-Macaulay i.e.,
        $(\mc{L}_1^{\vee})^{\otimes j-1}$ is almost full as an $\mo_{\wt{D}^{(1)}}$-module for all $1 \le j \le n$.
        In particular taking $j=1$, this implies $\mo_{D(1)}$ is Cohen-Macaulay.
        Recursively 
        if for some fixed $1 \le i < n$, $(\mc{L}_i^{\vee})^{\otimes j-i}$ is almost full for all $i \le j \le n$, 
        \eqref{eq:recur-4} implies that we have the short exact sequence (use the vanishing of $H^1$ as shown in Proposition \ref{prop:van}):
         \[0 \to \pi_*((\mc{L}_i^{\vee})^{\otimes j-i}) \to \pi_*((\mc{L}_i^{\vee})^{\otimes j-i-1}) \to \pi_*((\mc{L}_{i+1}^{\vee})^{\otimes j-i-1}) \to 0,\, \, \mbox{ for all } i+1 \le j \le n.\]
        By depth comparison in short exact sequences, this implies 
        $(\mc{L}_{i+1}^{\vee})^{\otimes j-i-1}$ is almost full for all $i+1 \le j \le n$.
        In particular taking $j=i+1$, this implies $\mo_{D(i+1)}$ is Cohen-Macaulay. Hence, for all $1 \le j \le n$, $D(j)$ is a 
        Cohen-Macaulay scheme. This proves the corollary.
    \end{proof}

\section{Application to the Bondal-Orlov conjecture}

\subsection{Setup}\label{se02}
Throughout this section we assume that $(X,x)$ is as in \S \ref{se01} with the additional assumption that $X$ is an isolated Gorenstein singularity and 
\[\pi^+:Y^+ \to X,\, \, \pi^-:Y^- \to X\]
be two small resolutions of $X$ related by a flop i.e., there exists a very ample invertible sheaf $\mc{L}^+$ on $Y^+$
such that the strict transform $\mc{L}^-$ of $\mc{L}^+$ in $Y^-$ is anti-ample i.e., the dual $(\mc{L}^-)^{\vee}$ is very ample. 
Denote by $E^+$ (resp. $E^-$) the exceptional loci of the small resolutions $\pi^+$ (resp. $\pi^-$).
Denote by $n^+:=\dim(E^+)$ and $n^-:=\dim(E^-)$.

\subsection{Tilting bundle and generator}
Setup as in \S \ref{se02}. Let $\mc{E}$ be a locally free sheaf on $Y^+$. Recall, that $\mc{E}$ is called \emph{generator of} $D^b(Y^+)$
if \[\mb{R}\mr{Hom}_{Y^+}(\mc{E}, \mc{K})=0\, \mbox{ for some } \mc{K} \in D^b(Y^+) \Rightarrow \, \mc{K}=0. \]
The same definition holds for a locally-free generator on $Y^-$.
Recall, the following useful result of Van den Bergh \cite{van1}:

\begin{lem}\label{lem:berg}
    Let $n^+=\dim(E^+)$ and $n^-:=\dim(E^-)$. Then, 
    \[\mc{E}^+:= \bigoplus_{i=0}^{n^+} (\mc{L}^+)^{\otimes i}\, \mbox{ and } (\mc{E}^-)^{\vee}:=\bigoplus_{i=0}^{n^-} (\mc{L}^-)^{\otimes -i}\]
    are generators of $D^b(Y^+)$ and $D^b(Y^-)$, respectively (here $(-)^{\otimes -i}$ means $((-)^{\vee})^{\otimes i}$). 
\end{lem}

\begin{proof}
    See \cite[Lemma $3.2.2$]{van1} for a proof.
\end{proof}

Recall, that if $\mc{E}^+$ (resp. $\mc{E}^-$) in Lemma \ref{lem:berg} is also a tilting locally-free sheaf, then there 
are natural equivalences of derived categories:
\[\mb{R}\mr{Hom}_{Y^+}(\mc{E}^+,-):D^b(Y^+) \xrightarrow{\sim} D^b(\mr{End}_{Y^+}(\mc{E}^+)) \mbox{ and }\]\[ \mb{R}\mr{Hom}_{Y^-}((\mc{E}^-)^{\vee},-):D^b(Y^-) \xrightarrow{\sim} D^b(\mr{End}_{Y^-}((\mc{E}^-)^{\vee}))\]
Note that, $\mc{E}^+$ and $(\mc{E}^-)^{\vee}$ are not always tilting bundles. We prove that if $\mc{L}^+$ is of CM-degree $n^+$, this holds true:

\begin{thm}\label{thm:orlov}
    Suppose that $\mc{L}^+$ has CM-degree $n^+$. If $n^+=n^-$, then $\mc{E}^+$ and $(\mc{E}^-)^{\vee}$ as in Lemma \ref{lem:berg} 
    are tilting generators. In particular, 
    we have an equivalence of derived categories:
    \[D^b(Y^+) \cong D^b(Y^-).\]
\end{thm}

\begin{proof}
   Using Lemma \ref{lem:berg}, it suffices to show that $\mc{E}^+$ (resp. $(\mc{E}^-)^{\vee}$) is a tilting bundle i.e.,
   \begin{equation}\label{eq:tilt}
       H^i(\Hc_{Y^+}(\mc{E}^+,\mc{E}^+))=0=H^i(\Hc_{Y^-}((\mc{E}^-)^{\vee},(\mc{E}^-)^{\vee}))\, \mbox{ for all }\, i >0.
   \end{equation}
  Since $\mc{L}^+$ has CM-degree $n^+$, Theorem \ref{thm:van-main} implies that $(\mc{L}^+)^{\otimes -j}$ is almost full for $1 \le j \le n^+$. 
  As $\pi^-$ is a small resolution, this implies 
  \[\pi^-_*((\mc{L}^-)^{\otimes -j})=\pi^+_*((\mc{L}^+)^{\otimes -j})\, \mbox{ is maximal Cohen-Macaulay}.\]
  Therefore $(\mc{L}^-)^{\otimes -j}$ is almost full for $1 \le j \le n^+$.
  Since $X$ is Gorenstein, Lemma \ref{lem:dual} implies that $(\mc{L}^-)^{\otimes j}$ is almost full for $1 \le j \le n^+$.
  As $\mc{L}^-$ is the dual of a very ample line bundle, the zero locus of a general section of 
  $(\mc{L}^-)^{\otimes -j}$ is irreducible, non-singular and intersects $E^-$ properly.
  Then, Proposition \ref{prop:van} implies that 
  \[ H^i((\mc{L}^+)^{\otimes -j})=0=H^i((\mc{L}^-)^{\otimes j})\, \mbox{ for all } i>0, 1 \le j \le n^+.\]
  Then, \cite[Lemma $2.1$]{toda2} implies that 
  \[ H^i((\mc{L}^+)^{\otimes j})=0=H^i((\mc{L}^-)^{\otimes -j})\, \mbox{ for all } i>0, 1 \le j \le n^+.\]
 Since $n^+=n^-$, this implies for all $i>0$, 
 \[H^i(\Hc_{Y^+}(\mc{E}^+,\mc{E}^+))=\bigoplus_{j=-n^+}^{n^+} H^i((\mc{L}^+)^{\otimes j})=0= \bigoplus_{j=-n^-}^{n^-} H^i((\mc{L}^-)^{\otimes j})=H^i(\Hc_{Y^-}((\mc{E}^-)^{\vee},(\mc{E}^-)^{\vee})).\]
 This proves the first part of the theorem.

    Since the exceptional loci $E^+$ and $E^-$ are of codimension at least $2$, we have an equivalence of categories of reflexive sheaves:
    \[\mr{Ref}(Y^+) \to \mr{Ref}(Y^-),\, \mbox{ given by } \mc{F} \mapsto i_*(\mc{F}|_U),\]
    where $U:=X\backslash \{x\}$ is the regular locus of $X$ and $i:U \to Y^-$ is the natural inclusion.
    Since $\mc{E}^+$ maps to $\mc{E}^-$ under this equivalence, this implies 
    \[\mr{Hom}_{Y^+}(\mc{E}^+, \mc{E}^+) \cong \mr{Hom}_{Y^-}(\mc{E}^-,\mc{E}^-) \cong \mr{Hom}_{Y^-}((\mc{E}^-)^{\vee}, (\mc{E}^-)^{\vee}),\]
    where the last equality follows from $\mc{E}^-$ is locally-free. Since $\mc{E}^+$ and $(\mc{E}^-)^{\vee}$ are tilting generators,
    this implies
    \[D^b(Y^+) \cong D^b(\mr{End}_{Y^+}(\mc{E}^+)) \cong D^b(\mr{End}_{Y^-}((\mc{E}^-)^{\vee})) \cong D^b(Y^-).\]
    This proves the theorem.
\end{proof}

\subsection{Example: Mukai flops}\label{sec:mukai}
Mukai flops relate cotangent bundle of $\mb{P}(V)$ and its dual. 
Derived equivalence between Mukai flops have been 
proved by Namikawa \cite{nami-der} and Kawamata \cite{kawa-1}, using Fourier-Mukai transform.
Here we give an alternative proof, using Theorem \ref{thm:orlov}.

Let $V = \mb{C}^{N}$, the cotangent bundle $Y^+ = \Omega_{\mb{P}(V)}$ can be written as
\[
    Y^+ := \bigl\{ (A,L)\in \mr{End}(V)\otimes \mb{P}(V)\ |\ A(V)\subset L, A^2 = 0\bigr\}.
\]
Then $Y^+$ is the crepant resolution of the nilpotent orbit closure $\ov{B(1)}$:
\[
\ov{B(1)} := \bigl\{ A\in \mr{End}(V)\ |\ \mr{rank}(A)\leq 1, A^2 = 0\bigr\}.
\]
Let $Y^- := \Omega_{\mb{P}(V^*)}$ be the cotangent bundle of the dual projective space. Then $Y^-$ is birational to $Y^+$ and gives 
another crepant resolution of $\ov{B(1)}$. 
The diagram
\[
    \xymatrix{
        Y^+ \ar[dr]^{\pi^+}  & & Y^- \ar[dl]_{\pi^-}\\
        & \ov{B(1)} &
    }
\]
is called the \emph{Mukai flop}.

Let $p: Y^+ \rightarrow \mb{P}(V)$ be the bundle projection map. We denote by $\co_{Y^+}(a) =p^*\co(a)$ and $\mo_{Y^-}(a)$ is the strict transform 
of $\mo_{Y^+}(a)$ in $Y^-$.

\begin{thm}
   Notations as above. Then, the invertible sheaf  $\mo_{Y^+}(1)$ is of CM-degree $N-1$. In particular, 
     \[\mc{E}^+:= \bigoplus_{i=0}^{N-1} \mo_{Y^+}(i)\, \mbox{ and } \mc{E}^-:=\bigoplus_{i=0}^{N-1} \mo_{Y^-}(-i)\]
   are tilting generators on $Y^+$ and $Y^-$, respectively. Moreover, we have an equivalence of derived categories $D^b(Y^+) \cong D^b(Y^-)$.
\end{thm}

\begin{proof}
 Using \cite[Lemma $3.2$]{ha17}, $H^i(\mo_{Y^+}(a))=0$ for all $i>0$ and $a \ge -N+1$.
 This implies 
 \begin{equation}\label{eq:van-11}
    H^{N-1}\left(\mc{E}^+ \otimes \left(\mc{E}^+\right)^{\vee}\right)=0
 \end{equation}
 Moreover, the Grothendieck spectral 
 sequence implies that \[\mr{Ext}^i(\mo_{Y^+}(a),\mo_{Y^+})=0, \, \mbox{ for }\, -N+1 \le a \le N-1\, \mbox{ and }\, i>0.\] 
 Note that, $\ov{B(1)}$ is a Gorenstein singularity. 
 Using Example \ref{example:ab}, we conclude that \[M:=\pi_*^+ \left(\bigoplus_{i=-N+1}^{N-1} \mo_{Y^+}(i)\right) 
  \mbox{ is a maximal Cohen-Macaulay }\, \mo_{\ov{B(1)}}-\mbox{module}.\] 
 Then, $M_0:=\pi^+_*\mc{E}^+$ is a tilting factor of $M$. Since $\pi^+$ is a small resolution, 
 \[((\pi^+)^*M_0)^{\vee \vee} \cong \mc{E}^+.\] 
 Since the dimension of the exceptional locus of $Y^+$ and $Y^-$ is $N-1\, =\, (\dim(\ov{B(1)}))/2$, Theorem \ref{thm:tf} along with \eqref{eq:van-11} implies that $\mc{E}^+$ is a tilting bundle.  
    Then, Corollary \ref{lem:tilt-2} implies that $\mo_{Y^+}(1)$ is of CM-degree $N-1$. This proves the first part of the theorem. Since the strict transform of $\co_{Y^+}(1)$ under Mukai flop is $\co_{Y^-}(-1)$ (see \cite[Lemma 1.3]{nami-der}),
    the remaining statements of the theorem follow from Theorem \ref{thm:orlov}. This proves the theorem.    
\end{proof}






\end{document}